\documentclass[leqno, 10.5pt]{amsart}

\usepackage{amsfonts,amsmath,amssymb,amsthm,amscd,enumerate,hyperref}
\usepackage[dvipsnames]{xcolor}

\newtheorem{thm}{Theorem}[section]

\newtheorem{lem}{Lemma}[section]
\newtheorem{cor}{Corollary}[section]

\newtheorem{prop}{Proposition}[section]

\theoremstyle{remark}
\newtheorem{rmk}{Remark}[section]

\numberwithin{equation}{section}

\theoremstyle{definition}

\newcommand{\C}{\ensuremath{\mathbb{C}}}

\newcommand{\R}{\ensuremath{\mathbb{R}}}

\newcommand{\CP}{\ensuremath{\mathbb{CP}}}
\newcommand{\RP}{\ensuremath{\mathbb{RP}}}
\newcommand{\rS}{\ensuremath{\mathbb{S}}}

\newcommand{\na}{\nabla}
\newcommand{\la}{\langle}

\newcommand{\ra}{\rangle}

\newcommand{\pa}{\partial}
\newcommand{\vphi}{\varphi}

\newcommand{\ka}{K\"ahler\ }

\newcommand{\tr}{\text{tr}}
\newcommand{\Vol}{\text{Vol}}

\newcommand{\trc}{\mathring{Rc}}

\newcommand{\obar}{\overline}

\newcommand{\KN}{\mathbin{\bigcirc\mspace{-15mu}\wedge\mspace{3mu}}}

\begin{document}

\title[4D gradient shrinking Ricci solitons with half PIC] {Four-dimensional complete gradient shrinking Ricci solitons with half positive isotropic curvature}
\author[Huai-Dong Cao And Junming Xie]{Huai-Dong Cao$^{\dag }$ And Junming Xie}

\address{Department of Mathematics, Lehigh University, Bethlehem, PA 18015}
\email{huc2@lehigh.edu; jux216@lehigh.edu}
\thanks{$^{\dag }$Research partially supported  by Simons Foundation Collaboration Grant \#586694}

\begin{abstract}
	In this paper, we investigate the geometry of $4$-dimensional complete gradient shrinking Ricci solitons with half positive isotropic curvature (half PIC) or half nonnegative isotropic curvature. Our first main result is a certain form of curvature estimates for such Ricci shrinkers, including a quadratic curvature lower bound estimate for noncompact ones with half PIC. As a consequence, we obtain a new and more direct proof of the classification result, first observed by Li-Ni-Wang \cite{Li-Ni-Wang:18}, for gradient shrinking K\"ahler-Ricci solitons of complex dimension two with nonnegative isotropic curvature. Moreover, based on a strong maximum principle argument, we classify 4-dimensional complete gradient shrinking Ricci solitons with half nonnegative isotropic curvature (except the half PIC case). Finally, we treat the half PIC case under an additional assumption on the Ricci tensor.
\end{abstract}
\maketitle


\section{Introduction}

A complete Riemannian manifold $(M^n,g)$ is called a {\em gradient shrinking Ricci soliton} if there exists a smooth (potential) function $f$ on $M^n$ such that the Ricci tensor $Rc$ of the metric $g$ satisfies the equation
\begin{equation} \label{eq:Riccishrinker}
	Rc+ \na^2f = \frac 1 2 g,
\end{equation}
where $\na^2 f$ denotes the Hessian of the potential function $f$. We usually normalize $f$, up to an additive constant, so that 
\begin{equation}  \label{eq: normalized f}
R+|\nabla f|^2=f, 
\end{equation}
where $R$ is the scalar curvature.

The subject of Ricci solitons was introduced by Hamilton \cite{Ha:88, Ha3:93} to study the formation of singularities in the Ricci flow. Ricci solitons are a natural extension of Einstein manifolds, and also self-similar solutions to Hamilton's Ricci flow arising as singularity models of the flow.  In particular, it follows from the work of Naber \cite{Naber:10} and Enders-M\"uller-Topping \cite{EMT:11} that rescaling limit singularity models of Type I maximal solutions on any compact manifold are necessarily nontrivial gradient shrinking Ricci solitons (see also the work of X. Cao and Q. S. Zhang \cite{C-Zq11}).

By the work of Hamilton \cite{Ha:88, Ha3:93}, it is known that any 2-dimensional complete gradient shrinking Ricci soliton is isometric to either $\rS^2$ or $\RP^2$, or the Gaussian soliton $\R^2$. Moreover, 3-dimensional gradient shrinking Ricci solitons have been completely classified through the works of Ivey \cite{Ivey:93}, Hamilton \cite{Ha3:93}, Perelman \cite{Perelman:03}, Naber \cite{Naber:10}, Ni-Wallach \cite{Ni-Wallach:08}, and Cao-Chen-Zhu \cite{Cao-C-Z:10} that they are isometric either to a finite quotient of $\rS^3$ or $\rS^2\times \R$, or to the Gaussian soliton $\R^3$.

While the classification of general gradient shrinking Ricci solitons in dimension four (or higher) is much more difficult and remains largely open, there has been a lot of progress on the classification of 4-dimensional gradient shrinking Ricci solitons with nonnegative curvature or special geometry. For example, the classification of compact shrinking Ricci solitons with positive curvature operator $Rm >0$ (or $Rm \ge 0$) follows from the well-known work of Hamilton \cite{Ha:86}. Furthermore, the same classification is valid under the weaker curvature assumption of 2-positive curvature operator\footnote{The classifications in dimension $n\geq 5$ follow from the work of B\"ohm-Wilking \cite{BW08}.} 
due to the work of H. Chen \cite{Chen 91}.  In the noncompact case, Naber \cite{Naber:10} classified $4$-dimensional complete noncompact gradient shrinking Ricci solitons with bounded and nonnegative curvature operator, $0\leq Rm\le C$. More recently, Munteanu and Wang \cite{M-W:17} further removed the bounded curvature condition and classified complete gradient Ricci shrinkers with $Rm\geq 0$ (for $n\geq 4$). 
For other classification results in dimension $n=4$ under various special curvature assumptions, such as locally conformally flat, half-conformally flat, harmonic Weyl, half harmonic Weyl, Bach-flat, of constant scalar curvature, etc., see, e.g.,  \cite{ENM:08, Ni-Wallach:08, Zhang2:09, Petersen-W:10, Cao-W-Z:11,  LR:11, M-S:13, Chen-Wang:15, Cao-Chen:13, WWW:18, LR:16, Cheng-Zhou 21} and the references therein. 

\smallskip
In this paper, we focus our attention on $4$-dimensional gradient shrinking Ricci solitons with half positive isotropic curvature.  Recall that a general Riemannian manifold $(M^n,g)$ of dimension $n\ge 4$ is said to have {\it positive isotropic curvature} if the Riemann curvature tensor $Rm=\{R_{ijkl}\}$ has the following property, 
\begin{equation} \label{eq:PIC}
R_{1313}+R_{1414}+R_{2323}+R_{2424}-2R_{1234}>0
\end{equation}
for any orthonormal 4-frame $\{e_1,e_2,e_3,e_4\}$.
The notion of positive isotropic curvature (PIC) was introduced by Micallef-Moore \cite{Micallef-M:88}, in which they proved that any compact simply connected $n$-dimensional Riemannian manifold with PIC is homeomorphic to a round sphere. Subsequently,  Micalleff-Wang \cite{MW:93} showed that the PIC condition is preserved under connected sums. Moreover, they  proved that the second Betti number of even dimensional compact, locally irreducible, manifolds with nonnegative isotropic curvature is at most one.  

For any oriented Riemannian 4-manifold $(M^4,g)$, it is well-known that the space of 2-forms $\wedge^2(M)$ admits the orthogonal decomposition 
$$\wedge^2(M) = \wedge^{+}(M) \oplus \wedge^{-}(M) $$
into the eigenspaces of the Hodge star operator $\star : \wedge^2(M) \to \wedge^2(M)$  of eigenvalues $\pm 1$. Smooth sections of $\wedge^+ (M)$ and $\wedge^-(M)$ are called {\it self-dual} and {\it anti-self-dual} 2-forms, respectively.  Accordingly, the Riemann curvature operator 
$$ Rm : \wedge^2(M) \to \wedge^2(M),$$
considered as a self-adjoint linear map, admits a block decomposition into  four pieces, 
\begin{equation} \label{eq:Rmdecomp}
	Rm = 
	\begin{pmatrix}
		A & B \\
		B^t & C
	\end{pmatrix}
	=
	\begin{pmatrix}
		\frac{R}{12} + W^+ & \mathring{Rc} \\
		\mathring{Rc} & \frac{R}{12} + W^-
	\end{pmatrix},
\end{equation}
where $W^{\pm}$ denote the self-dual and anti-self-dual part of the Weyl tensor, and $\mathring{Rc}$ is the traceless Ricci part.  It turns out (see \cite{Ha:97} or Section 2.2 below)  that $(M^4,g)$ has PIC  if and only if the $3\times 3$ matrices $A$ and $C$ in (\ref{eq:Rmdecomp}) are both {\it 2-positive} (i.e., the sum of the two least eigenvalues is positive).  
In \cite{Ha:97}, Hamilton showed that the PIC condition is preserved under the Ricci flow in dimension four\footnote{Later, this was proved in all dimensions $n\geq 5$ by Brendle-Schoen \cite{Brendle-S:09} and Nguyen \cite{Nguyen:10} independently, and this property played an essential role in Brendle-Schoen's proof of the long standing $1/4$-pinching differentiable sphere theorem.} and initiated the investigation of 4-dimensional Ricci flow with surgery under the PIC assumption; see also the work of Chen-Zhu \cite{Chen-Zhu:06}. Subsequently,  by using the Ricci flow with surgery developed in \cite{Ha:97, Chen-Zhu:06}, Chen-Tang-Zhu \cite{Chen-T-Zhu:12} completely classified compact $4$-manifolds with PIC up to diffeomorphisms.

For Einstein manifolds with PIC, Brendle \cite{Brendle:10} proved that they must be a finite quotient of the unit (round) sphere $\rS^n$, up to scaling.  Also, he proved that Einstein manifolds with nonnegative isotropic curvature are locally symmetric. On the other hand, for Ricci solitons, it was proved recently by Li-Ni-Wang \cite{Li-Ni-Wang:18} that any 4-dimensional complete gradient shrinking Ricci soliton with PIC is a finite quotient of either $\rS^4$ or $\rS^3\times \R$. Moreover, in dimension $n=4$, Richard-Seshadri \cite{Richard-Seshadri:16} extended Brendle's result by showing that a compact oriented Einstein 4-manifold with half PIC is isometric to $\rS^4$ or $\CP^2$. 

By definition, an oriented 4-manifold $(M^4, g)$ is said to have {\it half positive isotropic curvature} (half PIC) if either the matrix $A$ or the matrix $C$ in \eqref{eq:Rmdecomp} is 2-positive. Similarly, $(M^4, g)$ has {\it half nonnegative isotropic curvature} if either $A$ or $C$ is weakly 2-positive (i.e., 2-nonnegative).  

In this paper, we shall investigate the geometry of four-dimensional gradient shrinking Ricci solitons with half PIC or half nonnegative isotropic curvature. Throughout the paper, we shall assume the 4-manifold $M^4$ is oriented. Our first main result is the following curvature lower bound estimates. 

\begin{thm} \label{thm:quadraticofC1}
	Let $(M^4, g, f)$ be a 4-dimensional complete gradient shrinking Ricci soliton. 

\smallskip
\begin{enumerate}
		\item[(a)] If $(M^4, g, f)$ has half nonnegative isotropic curvature, then either the matrix $A$ is nonnegative (definite) or the matrix $C$ is nonnegative.\footnote{We obtained this result in Spring 2018. In \cite{Cho-Li:20}, Cho and Li proved the same result, namely Theorem \ref{thm:quadraticofC1} \!(a), independently.}

\smallskip
               \item[(b)] If $(M^4, g, f)$ has half positive isotropic curvature, then either $A>0$  or $C>0$. Moreover, if $M^4$ is noncompact then there exists some constant $K>0$ such that the smallest eigenvalue $A_1$ of $A$, or $C_1$ of $C$, satisfies the estimate 
$$ A_1 \geq \frac{K}{f} , \qquad 
{\mbox{or}} \ \ C_1 \geq \frac{K}{f}.$$
\end{enumerate}
 \end{thm}

As a consequence of Theorem \ref{thm:quadraticofC1}, we obtain the following classification result, first observed by Li-Ni-Wang \cite{Li-Ni-Wang:18} (see Corollary 3.1 in \cite{Li-Ni-Wang:18}), for gradient  K\"ahler-Ricci shrinkers of complex dimension two with nonnegative isotropic curvature.

\begin{cor} \label{thm:kahler}
	Let $(M^4,g,f)$ be a complete gradient shrinking K\"ahler-Ricci soliton of complex dimension two.

	\begin{enumerate}
		\item[(i)] If $(M^4,g,f)$ has half positive isotropic curvature, then it is, up to automorphisms, the complex projective space $\CP^2$. 

	    \smallskip	
		\item[(ii)] If $(M^4,g,f)$ has nonnegative isotropic curvature, then it is, up to automorphisms, one of the following: the complex projective space $\CP^2$, the product $\CP^1 \times \CP^1$, the cylinder $\CP^1 \times \C$, or the Gaussian soliton on $\C^2$.

	\end{enumerate}
\end{cor}

\begin{rmk}
We note that Li-Ni-Wang \cite{Li-Ni-Wang:18} used a Bony type strong maximum principle for degenerate elliptic equations from Brendle-Schoen \cite{Brendle-S:08}.
In our case, as we shall see later, it follows easily from Theorem \ref{thm:quadraticofC1} that any gradient shrinking K\"ahler-Ricci soliton of complex dimension two with nonnegative isotropic curvature must have nonnegative curvature operator $Rm\ge 0$. This leads to a new and more direct proof of Corollary \ref{thm:kahler}. 
\end{rmk}

\medskip
By Theorem \ref{thm:quadraticofC1} and a strong maximum principle argument, we also obtain

\begin{thm} \label{thm:dichotomy}
	Let $(M^4, g, f)$ be a 4-dimensional complete gradient shrinking Ricci soliton with half nonnegative isotropic curvature. Then, $(M^4, g, f)$ either has half positive isotropic curvature, or is a gradient shrinking K\"ahler-Ricci soliton, or is isometric to the Gaussian soliton $\R^4$ or a finite quotient of  $\rS^2 \times \rS^2$ or $\rS^2 \times \R^2$. 
\end{thm}

We note that  a K\"ahler surface is necessarily of half nonnegative isotropic curvature, see the special $Rm$ block decomposition formula \eqref{eq:RmforKahlersurfaces}.  It is also well-known that, besides del Pezzo/Fano K\"ahler-Einstein surfaces, the only compact complex two dimensional shrinking K\"ahler-Ricci solitons are either the $U(2)$-invariant Cao-Koiso soliton on $\CP^2 \# \obar{\CP^2}$ \cite{Cao, Koiso}  or the toric Wang-Zhu soliton on $\CP^2 \# 2\obar{\CP^2}$  \cite{WZ}. Moreover,  the very recent work of Conlon-Deruelle-Sun \cite{CDS:19}, Bamler-Cifarelli-Conlon-Deruelle \cite{BCCD:22}, and Li-Wang \cite{LW:23} have led to a complete classification of complete noncompact gradient shrinking K\"ahler-Ricci solitons in complex dimension two; namely, besides the Gaussian soliton $\C^2$ and the cylinder $\CP^1 \times \C$, they are either the $U(2)$-invariant FIK soliton constructed by Feldman-Ilmanen-Knopf \cite{FIK} on the blowup of $\C^2$ at the origin, or the toric BCCD soliton constructed recently by Bamler-Coifarelli-Conlon-Deruelle \cite{BCCD:22} on the blowup of $\CP^1 \times \C$ at one point.

Combining Theorem \ref{thm:dichotomy}  with the above mentioned facts, we immediately have the following classification result for 4-dimensional complete gradient shrinking Ricci soliton with half nonnegative isotropic curvature. 

\begin{cor} \label{thm:classification}
	Let $(M^4, g, f)$ be a 4-dimensional complete gradient shrinking Ricci soliton with half nonnegative isotropic curvature. Then, either
\begin{enumerate}
		\item[(i)] $(M^4, g, f)$ has half positive isotropic curvature, or 

\smallskip
\item[(ii)] $(M^4, g, f)$ is isometric to the Gaussian soliton $\R^4$ or a finite quotient of  $\rS^2 \times \rS^2$ or $\rS^2 \times \R^2$, or

\smallskip
\item[(iii)] $(M^4, g, f)$ is, up to automorphisms, one of the following: a closed del Pezzo surface with its unique K\"aher-Einstein or K\"ahler-Ricci soliton metric, the FIK soliton on the blowup of $\C^2$ at the origin,  the BCCD soliton on the blowup of $\CP^1 \times \C$ at one point. 

\end{enumerate}
\end{cor}

Our last result treats a special case of the half PIC case.

\begin{thm} \label{thm:main}
	Let $(M^4,g,f)$ be a 4-dimensional complete gradient shrinking Ricci soliton such that its Ricci tensor has an eigenvalue with multiplicity 3. 
	\begin{enumerate}
		\item[(a)] If $(M^4,g,f)$  has half positive isotropic curvature, then it is either isometric to $\rS^4$, $\CP^2$, or a finite quotient of $\rS^3 \times \R$.
	
		\smallskip	
		\item[(b)] If $(M^4,g,f)$  has half nonnegative isotropic curvature, then it is either isometric to  $\rS^4$, $\CP^2$, $\CP^1 \times \CP^1$, the Gaussian soliton $\R^4$, or a finite quotient of $\rS^3 \times \R$.
	\end{enumerate}
\end{thm}

\begin{rmk} The assumption on the Ricci tensor having an eigenvalue with multiplicity 3 is a technical one. We expect that this extra condition could be removed to achieve a full classification.
\end{rmk}

We would like to point out that the condition of half PIC or half weakly PIC is preserved by the Ricci flow  in dimension four; see the proof of Theorem B1.2 in \cite{Ha:97}\footnote{See also \cite{Richard-Seshadri:16} for a different proof.}.  As Hamilton mentioned more than once, an important problem is to understand formation of singularities of the Ricci flow on compact $4$-manifolds with half PIC and use the Ricci flow to investigate the topology of such manifolds. Our work is only an initial attempt in this direction.

\smallskip
The paper is organized as follows. In Section 2, we fix the notation and collect several known facts about the curvature decomposition in dimension 4, especially for K\"ahler surfaces, and about some fundamental properties of gradient shrinking Ricci solitons that will be used in the proof of Theorem \ref{thm:quadraticofC1} and Theorem  \ref{thm:main}. In Section 3, inspired by the work of B.-L. Chen \cite{ChenBL:09} and the work of Chow-Lu-Yang \cite{CLY11}, we prove Theorem \ref{thm:quadraticofC1} on the curvature lower bound estimates. As applications of Theorem \ref{thm:quadraticofC1}, Corollary \ref{thm:kahler} and Theorem \ref{thm:dichotomy} will also be shown. 
Finally, in Section 4, we carry out the proof of Theorem \ref{thm:main}. Unlike the pointwise maximum principle arguments used in the proof of Theorem \ref{thm:quadraticofC1}, this is done by using a version of the Yau-Naber Liouville Theorem from \cite{Petersen-W:10}. \\

\noindent {\bf Acknowledgements.} We would like to thank Professor Richard Hamilton and Professor Lei Ni for their interests in our work. We also like to thank Dr. Jiangtao Yu for discussions and the anonymous referee for helpful comments. 

\bigskip
\section{Preliminaries}
In this section, we fix the notation  and recall the concept of (half) positive isotropic curvature. Moreover, we collect several known results about the curvature operator decomposition for oriented Riemannian 4-manifolds, especially for K\"ahler surfaces, as well as some fundamental facts about gradient shrinking Ricci solitons. Throughout the paper, we denote by
$$Rm=\{R_{ijkl}\}, \quad Rc=\left(R_{ik}\right),\quad R $$
the Riemann curvature tensor, the Ricci tensor, and the scalar curvature of the metric $g$ either in local coordinates or local orthonormal frame, respectively.

\subsection{Positive isotropic curvature}

Let $(M^n,g)$ be an $n$-dimensional Riemannian manifold. For any point $p\in M$, let $Rm:\wedge^2T_pM \rightarrow \wedge^2T_pM$ be the curvature operator. We can complexify the tangent space $T_pM$ and the space of two forms $\wedge^2T_pM$ to get $T_pM \otimes \C$ and $\wedge^2T_pM \otimes \C$, respectively, and consider the $\C$-linear extension of $Rm$ to $\wedge^2T_pM \otimes \C$. We may extend the Riemannian inner product on $T_pM$ either as a $\C$-bilinear form $(\cdot,\cdot)$ or a Hermitian inner product $\la \cdot, \cdot \ra$ on $T_pM \otimes \C$. The latter extension gives rise to a Hermitian metric, again denoted by $\la \cdot, \cdot \ra$, on $\wedge^2T_pM \otimes \C$. For any complex plane $\sigma$ spanned by a unitary basis $\{v,w\} \in T_pM \otimes \C$, we define the {\em complex sectional curvature} of $\sigma$ as
\begin{equation} \label{eq:complexsectionalcur}
	K_{\C}(\sigma):=\la {Rm} (v\wedge w), (v\wedge w) \ra.
\end{equation}

We say that a vector $v\in T_pM\otimes \C$ is {\em isotropic} if $(v,v)=0$, and $\sigma$ is an {\em isotropic complex plane} if every vector in it is isotropic. The Riemannian manifold $(M^n,g)$ is said to have {\em positive isotropic curvature} (PIC) if $K_{\C}(\sigma) >0$, and {\em nonnegative isotropic curvature} (NNIC) if $K_{\C}(\sigma) \geq0$, whenever $\sigma$ is an isotropic complex plane.

If we decompose a complex vector $v$ into its real and imaginary parts by $v=x+iy$, then the condition of $(v,v)=0$ is equivalent to $g(x,x)=g(y,y)$ and $g(x,y)=0$. Thus, a complex plane $\sigma=\text{span}\{v,w\}$ of $T_pM \otimes \C$ is isotropic if and only if there exist orthonormal vectors $\{e_1,e_2,e_3,e_4\}$ such that the unitary basis $\{v,w\} $ can be expressed as
$$ \sqrt{2}v=e_1+ie_2, \quad \sqrt{2}w=e_3+ie_4. $$
By expanding  (\ref{eq:complexsectionalcur}), we see that $(M^n,g)$ has PIC (or NNIC) if and only if
\begin{gather*}
	R_{1313}+R_{1414}+R_{2323}+R_{2424}- 2R_{1234}>0 \\
	(or \ R_{1313}+R_{1414}+R_{2323}+R_{2424}- 2R_{1234} \geq 0)
\end{gather*}
for any orthonormal 4-frame $\{e_1,e_2,e_3,e_4\}$.

\subsection{Curvature decomposition of four-manifolds}

For any 4-dimensional oriented Riemannian manifold $(M^4,g)$, by using the Hodge star operator, we have the decomposition of the bundle of 2-forms 
\begin{equation} \label{eq:decompof2forms}
\wedge^2(M) = \wedge^{+}(M) \oplus \wedge^{-}(M), 
\end{equation}
where $\wedge^+ (M)$ consists of {\em self-dual} 2-forms and $\wedge^- (M)$ {\em anti-self-dual} 2-forms.   We say $(M^4, g)$ has {\em half positive isotropic curvature} (half PIC) if the complex sectional curvature $K_{\C}(\sigma) >0$, and {\em half nonnegative isotropic curvature} (half NNIC) if $K_{\C}(\sigma) \geq 0$, either for all isotropic complex planes $\sigma \subset \wedge^+(M) \otimes \C $ or for all $\sigma \subset \wedge^-(M)\otimes \C $, respectively.  According to the decomposition \eqref{eq:decompof2forms} for $\wedge^2(M)$, we have the following corresponding decomposition of the curvature operator:
\begin{equation} \label{eq:CODecomposition}
	Rm = 
	\begin{pmatrix}
		A & B \\
		B^t & C
	\end{pmatrix}
	=
	\begin{pmatrix}
		\frac{R}{12} + W^+ & \mathring{Rc} \\
		\mathring{Rc} & \frac{R}{12} + W^-
	\end{pmatrix},
\end{equation}
where $W^{\pm}$ denote the self-dual and anti-self-dual Weyl curvature tensors, respectively, and $\mathring{Rc}$ denotes the traceless Ricci tensor\footnote{More precisely, the operator $B:\wedge^{-}(M)\to \wedge^{+}(M)$ is given by $\mathring{Rc}\KN g$, the Kulkarni-Nomizu product of $\mathring{Rc}$ and $g$. In particular, $B$ is identically zero when $(M^4, g)$ is Einstein.}. 

Under the decomposition (\ref{eq:decompof2forms}), for any $p\in M^4$, we may choose a basis for $\wedge_p^+(M)$ and for $\wedge_p^-(M)$ as follows:
\begin{equation*}
	\begin{split}
		\vphi^+_1 = \frac{1}{\sqrt{2}}(e_1\wedge e_2 + e_3\wedge e_4),  \\
		\vphi^+_2 = \frac{1}{\sqrt{2}}(e_1\wedge e_3 + e_4\wedge e_2),  \\
		\vphi^+_3 = \frac{1}{\sqrt{2}}(e_1\wedge e_4 + e_2\wedge e_3),  \\
	\end{split}
	\quad \quad
	\begin{split}
		\vphi^-_1 = \frac{1}{\sqrt{2}}(e_1\wedge e_2 - e_3\wedge e_4),  \\
		\vphi^-_2 = \frac{1}{\sqrt{2}}(e_1\wedge e_3 - e_4\wedge e_2),  \\
		\vphi^-_3 = \frac{1}{\sqrt{2}}(e_1\wedge e_4 - e_2\wedge e_3),  \\
	\end{split}
\end{equation*}
where $\{e_1,e_2,e_3,e_4\}$ is any positively oriented orthonormal basis of $T_pM$. Here, we have used the metric $g$ to identify $T_pM$ and $T_p^*M$. The inner product on 2-forms is defined by
\begin{equation} \label{eq:innerproducttwoforms}
	\la X\wedge Y, V\wedge W \ra = \la X,V\ra \la Y,W\ra - \la X,W\ra \la Y,V\ra .
\end{equation}

Observe that, for the matrices $A$ and $C$ in (\ref{eq:CODecomposition}), 
\begin{equation*}
	\begin{split}
		A_{11}=\frac{1}{2}\left( R_{1212}+R_{3434}+2R_{1234} \right), \\
		A_{22}=\frac{1}{2}\left( R_{1313}+R_{4242}+2R_{1342} \right), \\
		A_{33}=\frac{1}{2}\left( R_{1414}+R_{2323}+2R_{1423} \right), \\
	\end{split}
	\quad \quad
	\begin{split}
		C_{11}=\frac{1}{2}\left( R_{1212}+R_{3434}-2R_{1234} \right), \\
		C_{22}=\frac{1}{2}\left( R_{1313}+R_{4242}-2R_{1342} \right), \\
		C_{33}=\frac{1}{2}\left( R_{1414}+R_{2323}-2R_{1423} \right). \\
	\end{split}
\end{equation*}
As noted in  \cite{Ha:97}, it follows that \[ K_{\C}(\sigma) =\frac 1 2(A_{22}+A_{33}) \] for a positively oriented basis $\{e_1,e_2,e_3,e_4\}$. On the other hand, if the basis had the opposite orientation, one would get 
\[ K_{\C}(\sigma) =\frac 1 2(C_{22}+C_{33}).\] Also, by the first Bianchi identity, $\tr A=\tr C=\frac   R 4.$

For the matrix $B$, we have
\begin{equation} \label{matrixB}
	\begin{split}
		B_{11}=\frac{1}{2}\left( R_{1212}-R_{3434} \right),  \\
	        B_{22}=\frac{1}{2}\left( R_{1313}-R_{4242} \right),  \\
		B_{33}=\frac{1}{2}\left( R_{1414}-R_{2323} \right), \\
	\end{split}
	\quad or \quad
	\begin{split}
		B_{11}=\frac{1}{4}\left( R_{11}+R_{22}-R_{33}-R_{44} \right), \\
		B_{22}=\frac{1}{4}\left( R_{11}+R_{33}-R_{44}-R_{22} \right), \\
		B_{33}=\frac{1}{4}\left( R_{11}+R_{44}-R_{22}-R_{33} \right), \\
	\end{split}
\end{equation}
and
\begin{equation*}
	B_{12}=\frac{1}{2}\left(R_{1213}+R_{3413}-R_{1242}-R_{3442}\right) =\frac{1}{2}\left(R_{23}-R_{14}\right), etc.
\end{equation*}
In fact, 
\begin{equation} \label{fullmatrixB}
	\begin{split}
		B = \frac{1}{2}
		\begin{pmatrix}
			R_{1212}-R_{3434} & R_{23}-R_{14} & R_{24}+R_{13} \\
			R_{23}+R_{14} & R_{1313}-R_{2424} & R_{34}-R_{12} \\
			R_{24}-R_{13} & R_{34}+R_{12} & R_{1414}-R_{2323}
		\end{pmatrix}.
	\end{split}
\end{equation}
Clearly, if we choose a frame  $\{e_1,e_2,e_3,e_4\}$ such that the Ricci tensor $Rc$ is diagonal then the matrix $B$ is also diagonal. In particular, $B$ is identically zero when $(M^4, g)$ is Einstein. Note that we also have the following expression of the traceless Ricci tensor in terms of the matrix $B$:
\begin{equation} \label{traceless Rc&B}
	\begin{split}
		\mathring{Rc} = 
		\begin{pmatrix}
			B_{11}+B_{22}+B_{33} & B_{32}-B_{23} & B_{13}-B_{31} & B_{21}-B_{12} \\
			B_{32}-B_{23} & B_{11}-B_{22}-B_{33} & B_{21}+B_{12} & B_{13}+B_{31} \\
			B_{13}-B_{31} & B_{21}+B_{12} & B_{22}-B_{11}-B_{33} & B_{23}+B_{32} \\
			B_{21}-B_{12} & B_{13}+B_{31} & B_{23}+B_{32} & B_{33}-B_{11}-B_{22}
		\end{pmatrix}.
	\end{split}
\end{equation}
Now, let us denote by $$A_1\leq A_2 \leq A_3 \qquad {\mbox{and}} \qquad  C_1 \leq C_2 \leq C_3$$ the eigenvalues of $A$ and $C$, respectively  and
$$a_1\leq a_2 \leq a_3 \qquad {\mbox{and}} \qquad  c_1 \leq c_2 \leq c_3$$ the eigenvalues of $W^+$ and $W^-$, respectively. Also let $0\leq B_1 \leq B_2 \leq B_3$ be the singular eigenvalues of $B$ and $\lambda_1 \leq \lambda_2 \leq \lambda_3 \leq \lambda_4$ be the eigenvalues of $\mathring{Rc}$. 

By choosing a suitable basis of $\wedge^+(M)$ and of $\wedge^-(M)$, we may assume
\begin{equation*}
	A = 
	\begin{pmatrix}
		\frac{R}{12} +a_1 & 0 & 0 \\
		0 &\frac{R}{12} +a_2 & 0 \\
		0 & 0 & \frac{R}{12} +a_3
	\end{pmatrix}
	,\quad
	C = 
	\begin{pmatrix}
		\frac{R}{12} +c_1 & 0 & 0 \\
		0 &\frac{R}{12} +c_2 & 0 \\
		0 & 0 & \frac{R}{12} +c_3
	\end{pmatrix}.
\end{equation*}

On the other hand, it is well known that 

\begin {itemize} 

\smallskip
\item  PIC is equivalent to $A_1+A_2 >0$ and $C_1+C_2>0$;

\smallskip
\item  NNIC is equivalent to $A_1+A_2 \geq 0$ and $C_1+C_2\geq 0$;

\smallskip
\item Half PIC is equivalent to either $A_1+A_2 >0$ or $C_1+C_2>0$;

\smallskip
\item  Half NNIC is equivalent to either $A_1+A_2 \geq 0$ or $C_1+C_2\geq 0$.

\end {itemize} 

\begin{lem} \textup{\bf (\cite{Li-Ni-Wang:18})} 
	Under the curvature operator decomposition (\ref{eq:CODecomposition}), we have the following algebraic identities for the various eigenvalues defined above:  
	\begin{gather}
		\sum_{i=1}^{4}\lambda_i^3 = 24\det B, \label{s2l2.7e3} \\
		\sum_{i=1}^{3}c_i^3 = 3c_1c_2c_3, \label{s2l2.7e2} \\
		\sum_{i=1}^{3}b_i^2 = \sum_{i=1}^{3}\tilde{b}_i^2 = \frac{1}{4}\sum_{i=1}^{4}\lambda_i^2, \label{s2l2.7e1}
	\end{gather}
	where $b_i^2 = \sum_{j=1}^{3}B_{ij}^2$ and $\tilde{b}_i^2 = \sum_{j=1}^{3}B_{ji}^2$.
\end{lem}

For the reader's convenience and the sake of completeness, we provide a proof here.

\begin{proof}
	First of all, by the relation between $\mathring{Rc}$ and $B$ in (\ref{traceless Rc&B}), we may assume
	\begin{gather} 
		\lambda_1 = B_1 - B_2 - B_3,  \quad  \lambda_2 = B_2 - B_1 - B_3, \label{eq:tracelessRc&B1} \\
		\lambda_3 = B_3 - B_1 - B_2, \quad  \lambda_4 = B_1 + B_2 + B_3, \label{eq:tracelessRc&B2}
	\end{gather}
	where $0\leq B_1 \leq B_2 \leq B_3$ are the singular eigenvalues of $B$ and $\lambda_1 \leq \lambda_2 \leq \lambda_3 \leq \lambda_4$ are the eigenvalues of $\mathring{Rc}$. Then, by direct computations, we have
	\begin{gather*}
		\lambda_1^3+\lambda_2^3 = -2B_3(3B_1^2+B_2^2+B_3^2-6B_1B_2), \\
		\lambda_3^3+\lambda_4^3 = 2B_3(3B_1^2+B_2^2+B_3^2+6B_1B_2).
	\end{gather*}
	Thus, (\ref{s2l2.7e3}) follows by combining the above two identities.
	
	For (\ref{s2l2.7e2}), since $W^-$ is traceless (i.e., $c_1+c_2+c_3=0$)  it follows that $c_1=-c_2-c_3$. Hence, we have
	\begin{equation*}
		\begin{split}
			c_1^3+c_2^3+c_3^3 & = (-c_2-c_3)^3 +c_2^3+c_3^3 \\
			&= -c_2^3-3c_2^2c_3 - 3c_2c_3^2 - c_3^3 + c_2^3 +c_3^3 \\
			& = -3(c_2+c_3)c_2c_3 \\
			&=3c_1c_2c_3.
		\end{split}
	\end{equation*}
	
	Finally, by (\ref{eq:tracelessRc&B1}), (\ref{eq:tracelessRc&B2}) and direct computations, we obtain $4|B|^2=|\mathring{Rc}|^2$, which implies (\ref{s2l2.7e1}).
\end{proof}

For any K\"ahler surface with complex structure $J$, we can choose a positively oriented orthonormal basis by $\{e_1,Je_1,e_2,Je_2\}$ for the tangent bundle. Then we have a natural basis of $\wedge ^2(M) = \wedge^+(M) \oplus \wedge^-(M)$ for the \ka surface as follows:
\begin{equation} \label{eq:Kahlerbasis}
	\begin{split}
		\vphi^+_1 = \frac{1}{\sqrt{2}}(e_1\wedge Je_1 + e_2\wedge Je_2),  \\
		\vphi^+_2 = \frac{1}{\sqrt{2}}(e_1\wedge e_2 + Je_2\wedge Je_1),  \\
		\vphi^+_3 = \frac{1}{\sqrt{2}}(e_1\wedge Je_2 + Je_1\wedge e_2),  \\
	\end{split}
	\quad \quad
	\begin{split}
		\vphi^-_1 = \frac{1}{\sqrt{2}}(e_1\wedge Je_1 - e_2\wedge Je_2), \\
		\vphi^-_2 = \frac{1}{\sqrt{2}}(e_1\wedge e_2 - Je_2\wedge Je_1), \\
		\vphi^-_3 = \frac{1}{\sqrt{2}}(e_1\wedge Je_2 - Je_1\wedge e_2). \\
	\end{split}
\end{equation}

Using such a special basis and the \ka condition, together with the Bianchi identity, it is known that one has the following curvature operator decomposition for \ka surfaces:
\begin{equation} \label{eq:RmforKahlersurfaces}
	Rm = 
	\begin{pmatrix}
		A & B \\
		B^t & C
	\end{pmatrix}
	=
	\begin{pmatrix}
		\frac{R}{4} & 0 & 0 & d_1 & d_2 & d_3 \\
		0 & 0 & 0 & 0 & 0 & 0 \\
		0 & 0 & 0 & 0 & 0 & 0 \\
		d_1 & 0 & 0 \\
		d_2 & 0 & 0 & &C \\
		d_3 & 0 & 0 \\
	\end{pmatrix}.
\end{equation}
Indeed, first note that the K\"ahler condition implies the following extra curvature properties:
	\begin{equation} \label{eq:KahlerRm1}
		\begin{split}
			Rm(X,Y,Z,W) &=Rm(JX,JY,Z,W) \\
			&=Rm(X,Y,JZ,JW) =Rm(JX,JY,JZ,JW),
		\end{split}
	\end{equation}
	\begin{equation} \label{eq:KahlerRm2}
		Rm(X,JX,Y,JY) = Rm(X,Y,X,Y) + Rm(X,JY,X,JY)
	\end{equation}
and
	\begin{equation} \label{eq:KahlerRc}
		Rc(X,Y) = Rc(JX, JY)
	\end{equation}
for all complex vector fields $X,Y,Z,W$. Now, by using the special basis in (\ref{eq:Kahlerbasis}), 
	\begin{equation*}
		\begin{split}
			2Rm(\vphi_1^{+},\vphi_1^{+}) &=Rm(e_1\wedge Je_1+e_2\wedge Je_2, e_1\wedge Je_1+e_2\wedge Je_2) \\
			&=Rm(e_1,Je_1,e_1,Je_1) +2Rm(e_1,Je_1,e_2,Je_2) + Rm(e_2,Je_2,e_2,Je_2)  \\
			&=Rm(e_1,Je_1,e_1,Je_1) +2Rm(e_1,e_2,e_1,e_2) \\
			&\quad + 2Rm(e_1,Je_2,e_1,Je_2) + Rm(e_2,Je_2,e_2,Je_2)  \\
			&=\frac{R}{2},
		\end{split}
	\end{equation*}
	where we have used  (\ref{eq:KahlerRm2}) in the third equality and (\ref{eq:KahlerRm1}) in the last equality. Similarly, one can compute directly to find 
\[ Rm(\vphi_2^{+},\vphi_2^{+}) =Rm(\vphi_3^{+},\vphi_3^{+})=0, \quad {\mbox{and}} \quad  Rm(\vphi_i^{+},\vphi_j^{+})  =0 \ (i\neq j).\]
Finally, the special form of the matrix B in (\ref{eq:RmforKahlersurfaces}) follows from formula (\ref{fullmatrixB}) and the facts that 
\[ R_{1313}\equiv Rm (e_1, e_2, e_1, e_2) = Rm(Je_1, Je_2, Je_1, Je_2) \equiv R_{2424},\]
\[R_{23}\equiv Rc (Je_1, e_2)=Rc (-e_1, Je_2)\equiv -R_{14}, \]
\[R_{34}\equiv Rc (e_2, Je_2)=Rc (Je_2, -e_2)=-Rc (e_2, Je_2) \equiv -R_{34}, \]
etc.

\subsection{Some basic facts about gradient shrinking Ricci solitons}

First, we recall some basic identities of  gradient shrinking Ricci solitons satisfying equation (\ref{eq:Riccishrinker}).
\begin{lem} \textup{\bf (Hamilton \cite{Ha3:93})} \label{lem:Hamilton}
	Let $(M^n,g,f)$ be an n-dimensional  gradient shrinking Ricci soliton satisfying Eq. (\ref{eq:Riccishrinker}). Then
	$$ R + \Delta f = \frac{n}{2}, $$
	$$ \na_iR = 2R_{ij}\na_jf, $$
	$$ R + |\na f|^2 = f. $$
\end{lem}
Moreover, we have the following well-known differential identities on the curvatures $R$, $Rm$ and its three components $A, B, C$. They are the special case of the curvature evolution equations under the Ricci flow derived by Hamilton \cite{Ha:86}. 

\begin{lem} \textup{\bf (Hamilton \cite{Ha:86})} \label{lem:ellieqofsolitons}
	Let $(M^4,g,f)$ be a 4-dimensional  gradient shrinking Ricci soliton satisfying Eq. (\ref{eq:Riccishrinker}). Then
	\begin{gather*}
		\Delta_f R = R - 2|Rc|^2, \\
		\Delta_f Rm =Rm - 2(Rm^2+Rm^{\sharp}), \\
		\Delta_f A = A - 2(A^2 +2A^{\sharp} +BB^t), \\
		\Delta_f B = B - 2(AB +BC + 2B^{\sharp}), \\
		\Delta_f C = C - 2(C^2 +2C^{\sharp} +B^tB).
	\end{gather*}
Here, $\Delta_f=\Delta -\na f\cdot\na $ is the drift Laplacian, $C^2$ denotes the matrix square of $C$, and $C^{\sharp}$ is the transpose of the adjoint matrix of $C$, etc. 
\end{lem}

\begin{rmk}
	 Except for the first equation in Lemma \ref{lem:ellieqofsolitons}, the factor $2$ in the front of parentheses differs from the corresponding equations in \cite{Ha:86} due to our slightly different definition of the inner product on $\wedge^2(M)$ as given in (\ref{eq:innerproducttwoforms}).
\end{rmk}

Next, we state several fundamental facts about complete gradient shrinking Ricci solitons that we shall need later. 

\begin{lem} \textup{\bf (Chen \cite{ChenBL:09})} \label{lem:Chen}
	Let $(M^n,g,f)$ be an n-dimensional complete gradient shrinking Ricci soliton. Then it has nonnegative scalar curvature $R\geq 0$.
\end{lem}

\begin{rmk} \label{rmk:R>0}
	It was further observed by Pigola-Rimoldi-Setti \cite{Pigola-R-S:11} that either $R>0$ everywhere or the shrinking Ricci soliton is the Gaussian shrinking soliton on $\R^n$.
\end{rmk}

\begin{lem} \textup{\bf (Cao-Zhou \cite{Cao-Zhou:10})} \label{lem:Cao-Zhou}
	Let $(M^n,g,f)$ be an n-dimensional complete gradient shrinking Ricci soliton of dimension n and $p\in M$. Then there are positive constants $c_1$, $c_2$ and $C$ such that
	$$  \frac{1}{4}(d(x,p)-c_1)_+^2 \leq f(x) \leq \frac{1}{4}(d(x,p)+c_2)^2,  $$
	$$  \Vol(B_p(r)) \leq Cr^n.  $$
\end{lem}

\begin{lem} \textup{\bf (Munteanu-Sesum \cite{M-S:13})} \label{lem:M-S}
	Let $(M^n,g,f)$ be an n-dimensional complete gradient shrinking Ricci soliton. Then for any $\lambda >0$, we have
	$$  \int_M |Rc|^2 e^{-\lambda f} < \infty.  $$
\end{lem}

Finally, we shall need the following Yau-Naber Liouville type theorem, and also an extension of Hamilton's tensor maximum principle by Petersen-Wylie \cite{Petersen-W:10}.

\begin{lem} \textup{\bf (Yau-Naber Liouville Theorem \cite{Petersen-W:10})} \label{lem:Yau-NaberLiouville}
	Let $(M, g, h)$ be a smooth metric measure space with finite $h$-volume: $\int e^{-h} \ d\text{vol} < \infty$. If $u$ is a locally Lipschitz function in $L^2(e^{-h}d\text{vol}_g)$ which is bounded below and such that
	$$ \Delta_h u \geq 0 $$
	in the sense of barriers, then $u$ is a constant.
\end{lem}

\begin{lem} \textup{\bf (\cite{Petersen-W:10})} \label{lem:Tensormaxprinciple}
	Let $(M, g, h)$ be any smooth metric measure space with finite $h$-volume. Suppose $T$ is a symmetric 2-tensor on some (tensor) bundle such that $|T|\in L^2(e^{-h}d\text{vol}_g)$ and, for some constant $\rho >0$,
	$$ \Delta_h T = \rho T + \Phi(T) \quad \text{with}\quad g(\Phi(T)(s),s) \leq 0  \ \text{for any section} \ s.$$
Then $T$ is nonnegative and $\ker (T)$ is invariant under parallel translation.
\end{lem}

\bigskip
\section{The proof of Theorem \ref{thm:quadraticofC1} and Theorem \ref{thm:dichotomy}}

In this section, we shall prove Theorem \ref{thm:quadraticofC1}, Corollary \ref{thm:kahler},  and Theorem  \ref{thm:dichotomy}  as stated in the introduction.

First of all, we divide the proof of Theorem  \ref{thm:quadraticofC1} into proving the following two propositions; the first one applies not only to complete gradient shrinking Ricci solitons but also to complete {\it ancient solutions} of the Ricci flow. 

\begin{prop} \label{pro:positiveC1}
	Let $(M^4, g (t))$ be a 4-dimensional complete ancient solution.  

\smallskip
\begin{enumerate}
		\item[(a)]\footnote{We obtained this result in Spring 2018. In  \cite{Cho-Li:20}, Cho and Li observed
the same result, i.e., Proposition \ref{pro:positiveC1} \!(a), independently.}  If $g(t)$ has half nonnegative isotropic curvature, then either $A\geq 0$ or $C\geq  0$.

\smallskip
\item[(b)] If $g(t)$ has half positive isotropic curvature, then either $A>0$ or $C> 0$.  

\end{enumerate}
\end{prop}

\begin{proof} (a) Our proof will essentially follow a similar argument by B.-L. Chen in \cite{ChenBL:09}; see also the survey article by the first author \cite{Cao:10}.  

\smallskip
Suppose $g (t)$ is defined for $-\infty < t \leq T$ for some positive $T>0$. Without loss of generality, we may assume that matrix $C$ is weakly 2-positive, and $C_3 \ge C_2\ge C_1$ are the eigenvalues of $C$ so that $$C_1+C_2\geq 0 \qquad 
{\mbox{on}} \ M\times (-\infty, T].  $$  Then it follows that $C_3 \geq C_2 \geq 0$ on $M\times (-\infty, T]$ as well.
	
	Firstly, let us consider $g (t)$ on the finite time interval $[0,T]$. For any fixed point $x_0 \in M$, pick $r_0>0$ sufficiently small such that (for dimension $n=4$)
	\begin{equation} \label{Ricci bound}
|Rc|(x,t) \leq 3r_0^{-2}
\end{equation}
 for all $x\in B_t(x_0,r_0)$ and $t\in [0,T]$. Here, $B_t(x_0,r_0)$ denotes the geodesic ball of radius $r_0$ centered at $x_0$ and at time $t$. Then, for any constant $\alpha>
{40}Tr_0^{-2}+2$, we pick a constant $K_{\alpha}>0$ such that $C_1(x,0) \geq -K_{\alpha}$ on $B_0(x_0,\alpha r_0)$ at $t=0$.

We claim that there exists a universal constant $K_0>0$ such that
	\begin{equation} \label{eq:C1inequality}
		C_1(x,t) \geq \min \left\{ -\frac{2}{t+\frac{2}{K_{\alpha}}},-\frac{K_0}{\alpha^2 r_0^2} \right\},
	\end{equation}
	whenever $x\in B_t(x_0,\frac{3\alpha}{4}r_0)$ and $t\in[0,T]$.
	
	Indeed, let us take any smooth nonnegative, nonincreasing, cut-off function $\phi$ on $\R$, such that $\phi \equiv 1$ on $(-\infty,7/8]$ and $\phi\equiv 0$ on $[1,\infty)$, and consider the function
	$$ u(x,t) :=\phi \left( \frac{d_t(x_0,x)+5tr_0^{-1}}{\alpha r_0} \right)  C_1(x,t).$$
	Then, by direct computations, we obtain
	\begin{equation*}
		\begin{split}
			\left( \frac{\pa}{\pa t}-\Delta \right) u &= \frac{\phi^{\prime}C_1}{\alpha r_0} \left[ \left( \frac{\pa}{\pa t}-\Delta \right) d_t(x_0,x) + \frac{5}{r_0} \right] \\
			& \hspace{4mm} - \frac{\phi^{\prime \prime}C_1}{\alpha^2r_0^2} - 2\na \phi \cdot \na C_1 + \phi \left( \frac{\pa}{\pa t}-\Delta \right)C_1
		\end{split}		
	\end{equation*}
	at smooth points of the distance function $d_t(x_0,\cdot)$.
	
	Let $u_{\min}(t) = \min_M u(\cdot,t)$. If $u_{\min}(t_0) \leq 0$ for some $t_0$, and $u_{\min}(t_0)$ is achieved at some $x_1$ such that $ u(x_1,t_0) = u_{\min}(t_0) $, then $ \phi^{\prime}C_1(x_1,t_0) \geq 0 $. On the other hand, by (\ref{Ricci bound}) and Lemma 8.3(a) of Perelman \cite{Perelman:02} (see also Lemma 3.4.1 in \cite{Cao-Zhu:06}), 
	\begin{equation*} \label{eq:Perelman}
		\left( \frac{\pa}{\pa t}-\Delta \right) d_t(x_0,x) \geq -\frac{5}{r_0},
	\end{equation*}
	whenever $d_t(x_0,x) > r_0$, in the sense of support functions. Then, as long as $u_{\min}(t_0) \leq 0$, by applying the maximum principle and standard support function technique, we get
	\begin{equation*}
		\begin{split}			
			\left.\frac{d^-}{dt} \right|_{t=t_0} u_{\min} &{\ :=}\liminf_{h\searrow 0} \frac{u_{\min}(t_0+h) - u_{\min}(t_0)}{h}   \\
& \geq - \frac{\phi^{\prime \prime}C_1}{\alpha^2r_0^2} - 2\na \phi \cdot \na C_1 + \phi \left( \frac{\pa}{\pa t}-\Delta \right)C_1  \\
			& \geq - \frac{\phi^{\prime \prime}C_1}{\alpha^2r_0^2} +\frac{2 \phi^{\prime 2}C_1}{\alpha^2r_0^2 \phi} + 2\phi (C_1^2 + 2C_2C_3) \\
			& \geq \frac{C_1}{\alpha^2r_0^2} \left(  \frac{2\phi^{\prime 2}}{\phi} - \phi^{\prime \prime} \right) +2\phi C_1^2,
		\end{split}
	\end{equation*}
	where we have used the fact that $C_3 \geq C_2\geq 0$ in the last inequality. Therefore, it follows from 
	$$|2\phi^{\prime 2}/\phi - \phi^{\prime \prime}| \leq K_0\sqrt{\phi}$$ for some universal constant $K_0>0$ and the Cauchy-Schwarz inequality
	$$  \left| \frac{C_1K_0\sqrt{\phi}}{\alpha^2r_0^2} \right| \leq  \phi C_1^2 + \frac{K_0^2}{4 \alpha^4r_0^4} $$
	that
	\begin{equation*}
		\left.\frac{d^-}{dt} \right|_{t=t_0} u_{\min} \geq \frac{1}{2} u_{\min}^2(t_0) + \left( \frac{1}{4} u_{\min}^2(t_0) - \frac{K_0^2}{4 \alpha^4 r_0^4} \right) ,
	\end{equation*}
	provided $u_{\min}(t_0) \leq 0$. If $u_{\min} \leq -\frac{K_0}{\alpha^2 r_0^2}$, then we have
	$$ \left.\frac{d^-}{dt} \right|_{t=t_0} u_{\min} \geq \frac{1}{2} u_{\min}^2(t_0). $$
	Hence, by integrating the above inequality from $0$ to $t$, we get
	$$   u_{\min}(t) \geq \min \left\{ -\frac{2}{t+\frac{2}{K_{\alpha}}},-\frac{K_0}{\alpha^2 r_0^2} \right\}, \quad \text{whenever}\ x\in B_t(x_0, \frac{3\alpha}{4}r_0).$$
	This proves the claim, hence estimate (\ref{eq:C1inequality}) holds for $C_1(x, t)$ on $B_t(x_0,\frac{3\alpha}{4}r_0) \times  [0,T]$. 
	
	Now we consider $g(t)$ as an ancient solution on $M^4$. For any fixed $t\in (-\infty,T]$, we apply the above arguments on the interval $[s,t]$ with $s<t$ by translating the initial time of $s$ to $0$. Then, by (\ref{eq:C1inequality}), we have
	\begin{equation*}
		C_1(x,t) \geq \min \left\{ -\frac{2}{t-s+\frac{2}{K_{\alpha}}},-\frac{K_0}{\alpha^2 r_0^2} \right\},\quad \text{whenever}\ x\in B_t(x_0, \frac{3\alpha}{4}r_0).
	\end{equation*}
	For any fixed $t$, by taking $\alpha \rightarrow \infty$ and then $s \rightarrow -\infty $, we obtain $C_1\geq 0$ for any ancient solution on $M^4$. 
	
(b) Without loss of generality, we again assume the matrix $C$ is 2-positive, i.e., $$C_1+C_2> 0 \qquad 
{\mbox{on}} \ M\times (-\infty, T].  $$  Then it follows that $C_3 \geq C_2 > 0$ on $M\times (-\infty, T]$. 

We shall prove $C_1>0$ by contradiction. Assume $C_1(x_0,t_0) = 0$ at some point $(x_0,t_0)$. Then $C_1$ attains its minimum at $(x_0,t_0)$. Let $\eta\in \wedge_{x_0}^{-}(M)$ be a null eigenvector of $C$ such that, at $(x_0,t_0)$, $C(\eta,\eta)=C_1 (x_0,t_0)=0$. Extend $\eta$ to a local section (also denoted by $\eta$) in space and time by parallel translating along geodesics emanating from $x_0$ and independent of $t$. Then, at $(x_0,t_0)$, we have 

\begin{equation} \label{eq:e4}
	\begin{split}
		0 & \geq (\pa_t - \Delta)C_1 \\
               & \geq (\pa_t - \Delta) C(\eta, \eta) \\
		&= [(\pa_t - \Delta)C](\eta, \eta) \\
                & = 2(C^2+B^tB+2C^{\sharp})(\eta, \eta) \\
		&\geq 2C_1^2 + 2B_1^2 + 4C_2C_3 \\
		&>0
	\end{split}
\end{equation}
in the barrier sense. This is a contradiction. 
Hence $C_1>0$  on $M^4$. 

This completes the proof of Proposition \ref{pro:positiveC1}.
\end{proof}

Our second proposition only applies to $4$-dimensional complete noncompact gradient shrinking Ricci solitons. 

\begin{prop} \label{prop:C1K/f}
	Let $(M^4,g,f)$ be a 4-dimensional complete noncompact gradient shrinking Ricci soliton with half positive isotropic curvature, then  either $A_1 \geq \frac{K}{f}$ or $C_1 \geq \frac{K}{f}$ for some constant $K>0$ depending only on the geometry on a fixed large geodesic ball.
\end{prop}

\begin{proof} By Proposition \ref{pro:positiveC1}, we know that either $A>0$ or $C>0$.  Again, without loss of generality, we 
may assume $C>0$. We shall use a similar argument as in Chow-Lu-Yang \cite{CLY11}. 

By Lemma \ref{lem:ellieqofsolitons} and the half PIC assumption, we have $\Delta_f C_1 \leq C_1$ in the barrier sense. 
Now, for any fixed point $p\in M^4$, we consider the geodesic ball $B_p(r_0)$ of radius $r_0$ centered at $p$,  for some $r_0>0$ to be chosen later. Define
	$$ a:= \inf_{\pa B_p(r_0)}C_1 >0,$$
	and
	$$ u:=C_1 - af^{-1} -4af^{-2}. $$
	Then, on one hand, we have $u>0$ on $\pa B_p(r_0)$  for $r_0$ sufficiently large. On the other hand, as
	\begin{eqnarray*}
		\Delta_{f} (f^{-1}) &=& -\Delta_f(f)f^{-2} + 2|\na f|^2f^{-3} \\
		&=& (f-2)f^{-2} + 2|\na f|^2f^{-3} \\
		&\geq& f^{-1} -2f^{-2},
	\end{eqnarray*}
	and
	\begin{eqnarray*}
		\Delta_{f} (f^{-2}) &=& 2(f-2)f^{-3} + 6|\na f|^2f^{-4} \\
		&\geq& \frac{3}{2}f^{-2},
	\end{eqnarray*}
	on $M\setminus B_p(r_0)$, we have
	\begin{eqnarray*}
		\Delta_f u &\leq& C_1 - af^{-1} + 2af^{-2} - 6af^{-2} \\
		&=& u.
	\end{eqnarray*}
	
	We claim that $u\geq 0$ on $M\setminus B_p(r_0)$. If not, then there exists a point $x_0 \in M\setminus B_p(r_0)$ such that $u(x_0)<0$. Since $u>0$ on $\pa B_p(r_0)$ and $u\geq 0$ at infinity, we know that $u$ achieves its negative minimum at some point $p_0$ in the interior of $M\setminus B_p(r_0)$. Thus, by the maximum principle, at the point $p_0$, we have $0\leq \Delta_f u\leq u<0$, which is a contradiction.
	
	Therefore $u\geq 0$ on $M\setminus B_p(r_0)$, and there exists some constant $K$ such that $C_1\geq \frac{K}{f}$ on $M$. This completes the proof of  Proposition \ref{prop:C1K/f} and concludes the proof of Theorem \ref{thm:quadraticofC1}.
\end{proof}

As an application of  Theorem \ref{thm:quadraticofC1},  we now prove Corollary \ref{thm:kahler}. 
\begin{proof}  Let $(M^4,g,f)$ be a complete gradient shrinking K\"ahler-Ricci soliton of complex dimension two with nonnegative isotropic curvature. The key point is to show 
that $(M^4,g,f)$ must have nonnegative curvature operator $Rm\ge 0$. 

On one hand, by Lemma \ref{lem:ellieqofsolitons} and decomposition \eqref{eq:RmforKahlersurfaces} of the curvature operator $Rm$ for K\"ahler surfaces, we have
\begin{equation*}
	\Delta_f Rm =Rm - 2(Rm^2+Rm^{\sharp}) =Rm - 2\left( Rm^{2}+ 
	2\left(\begin{array}{cc}
		0 & 0 \\
		0 & C^{\#}
	\end{array}\right)
	\right).
\end{equation*}
Since $C\ge 0$ by Theorem \ref{thm:quadraticofC1}, it follows that the transpose of its adjoint $C^{\#}\geq 0$.  
 Hence $Rm^2+Rm^{\sharp} \geq 0$. 
On the other hand, note that
$$ |Rm|^2 \leq 2(|A|^2 + |C|^2 + |B|^2). $$
Moreover, nonnegative isotropic curvature implies that $|C|^2 \leq \frac{3}{16}R^2$, while the K\"ahler condition implies $|A|^2 = \frac{1}{16}R^2$. Therefore, together with identity (\ref{s2l2.7e1}), i.e., $|B|^2 = \frac{1}{4}|\trc|^2$, we have
$$ |Rm|^2 \leq \frac{5}{2}|Rc|^2. $$
Now, it follows that $Rm\geq 0$ by Lemma \ref{lem:M-S} and Lemma \ref{lem:Tensormaxprinciple}.

Now the classification (under half nonnegative isotropic curvature condition) follows from the work of Munteanu-Wang \cite{M-W:17} (Corollary 4) on shrinking Ricci solitons with nonnegative curvature operator, or from the work of  Ni \cite{Ni:05} on K\"{a}hler-Ricci solitons with nonnegative bisectional curvature. In addition, if $(M^4, g, f)$ has half PIC, then clearly $\C^2$, $\CP^1 \times \C$ and $ \CP^1 \times \CP^1$ are excluded. This finishes the proof of Corollary \ref{thm:kahler}.
\end{proof}

Next, we prove {\bf Theorem \ref{thm:dichotomy}}.

\begin{proof} Let $(M^4, g, f)$ be a 4-dimensional complete gradient shrinking Ricci soliton with half nonnegative isotropic curvature. Then, by Theorem \ref{thm:quadraticofC1}, we have either $A\ge 0$ or $C\ge 0$. Without loss of generality, we may assume $A\ge 0$.  

First of all, we claim that $\ker (A)$ is invariant under parallel translation.  Indeed, since $A\geq 0$ and its three eigenvalues $0\leq A_1\leq A_2\leq A_3$ has the sum $A_1+A_2+A_3=\frac{R}{4}$, we get $|A|^2\leq \frac {1}{16}{R^2}$. Combining this with Lemma \ref{lem:ellieqofsolitons} and Lemma \ref{lem:M-S}, and applying Lemma \ref{lem:Tensormaxprinciple}, we conclude that $\ker(A)$ is invariant under parallel translation.

Next, we show that if the holonomy group $\text{Hol}^{0} (M^4, g)$ is $\text{SO}(4)$, then $A>0$ hence $(M^4, g, f)$ has half PIC. We argue by contradiction. Suppose that there exist a point $p\in M^4$ and a self-dual bivector $\vphi_1 \in \wedge_{p}^+(M)$ such that
	$$ A (\vphi_1,\vphi_1) = 0. $$
	It is then clear that $\vphi_1$ is a null eigenvector corresponding to the smallest eigenvalue $A_1=0$. Now it is an elementary fact that, in dimension $n=4$, any self-dual 2-form $\vphi_1 \in \wedge_{p}^+(M)$ can be expressed as 
	\begin{equation*} \label{eq:vphi1}
		\vphi_1 = \frac{1}{\sqrt{2}}\left( e_1\wedge e_2 + e_3\wedge e_4\right)
	\end{equation*}
for some positively oriented orthonormal frame $\{e_1, e_2, e_3, e_4\}$ at $p$; see Lemma 6.1 in \cite{Derdzinski:00}. 
	Meanwhile, suppose $\vphi_3\in \wedge_{p}^+(M)$ is an eigenvector corresponding to the largest eigenvalue $A_3$. Then, using Lemma 6.1 in \cite{Derdzinski:00} again, we can find another positively oriented orthonormal frame $\{v_1, v_2, v_3, v_4\}$ at $p$ such that
	\begin{equation*} \label{eq:vphi3}
		\vphi_3 = \frac{1}{\sqrt{2}}\left( v_1\wedge v_2 + v_3\wedge v_4\right). 
	\end{equation*}
	Since $\text{Hol}^{0} (M^4, g)=\text{SO}(4)$, there exists a closed loop $\gamma$ based at $p$ such that 
	$$ v_i = P_{\gamma}e_i,\quad i=1, \cdots, 4,  $$
	where $P_{\gamma}$ denotes the parallel transport along $\gamma$. It then follows that $A_3=A(\vphi_3,\vphi_3)$ $= A(\vphi_1,\vphi_1)=A_1=0 $, since $\ker (A)$ is invariant under parallel translation. This would imply that the scalar curvature $R=4(A_1+A_2+A_3)=0$ at $p$. Then, by Remark \ref{rmk:R>0}, $(M^4,g,f)$ would be isometric to the Gaussian shrinking soliton on $\R^4$ which is a contradiction to the assumption that $\text{Hol}^{0} (M^4, g)=\text{SO}(4)$.
	
Now, if $(M^4,g,f)$ is locally reducible, then it must be either the Gaussian soliton on $\R^4$ or a finite quotient of either $\rS^2 \times \rS^2$, or $\rS^2 \times \R^2$, or  $\rS^3 \times \R$. On the other hand, if $(M^4,g,f)$ is irreducible and (locally) symmetric, then it must be of compact type because the scalar curvature $R>0$. But then $(M^4,g)$ must be either $\rS^4$ or $\CP^2$. Finally, if $(M^4,g,f)$ is irreducible and not isometric to a symmetric space then, by Berger's holonomy classification theorem, either $\text{Hol}^{0} (M^4, g)=\text{SO}(4)$ or $\text{Hol}^{0} (M^4, g)=\text{U}(2)$. If $\text{Hol}^{0} (M^4, g)=\text{U}(2)$, then $(M^4,g, f)$ is a gradient shrinking K\"ahler-Ricci soliton. On the other hand, if $\text{Hol}^{0} (M^4, g)=\text{SO}(4)$ then, from the above, we know that $(M^4,g,f)$ must have half PIC.   

This completes the proof of Theorem \ref{thm:dichotomy}.
\end{proof}

\section{The proof of Theorem \ref{thm:main}}

In this section, we prove Theorem \ref{thm:main} as stated in the introduction. By the half PIC (or half NNIC) assumption, without loss of generality, we may assume that the matrix $C$ is 2-positive (or weakly 2-positive), i.e., $C_1+C_2>0$ (or $C_1+C_2\geq0$).

We start by deriving a key differential inequality which will be used in the proof of Theorem \ref{thm:main}. 

\begin{lem} \label{lem:inequalityof|C|/R_1}
	Let $(M^4,g,f)$ be a 4-dimensional complete gradient shrinking Ricci soliton satisfying Eq. (\ref{eq:Riccishrinker}) and with $R>0$. Then,
	\begin{equation*}
\begin{split}
		\Delta_F\frac{|C|}{R} & \geq \frac{2|C|}{R^2}|Rc|^2 - \frac{1}{R|C|}\la 2(C^2 + B^tB + 2C^{\sharp}), C \ra\\
& \geq \frac{2}{R^2|C|}\left( \left( \frac{1}{4}R^2 \sum_{i=1}^{3}c_i^2 - 3R\sum_{i=1}^{3}c_i^3 \right) + 4\left( \sum_{i=1}^{3}c_i^2 \right) \left( \sum_{i=1}^{3} \tilde{b}_i^2 \right) - R \sum_{i=1}^{3}c_i\tilde{b}_i^2 \right), 
\end{split}	
\end{equation*}
	where $F = f - 2\log R$ and $\Delta_F = \Delta - \la \na F, \na \ra$.
\end{lem}

\begin{proof}
	First of all, by direct computations, we have
	\begin{equation} \label{3.l1e1}
		\Delta_f \frac{|C|}{R} = \frac{1}{R}\Delta_f |C| - \frac{|C|}{R^2}\Delta_f R - \frac{2}{R^2} \la \na |C|, \na R \ra  + \frac{2|C|}{R^3} \la \na R, \na R \ra.
	\end{equation}
	On the other hand, using Kato's inequality, we get
	\begin{equation} \label{3.l1e2}
		\begin{split}
			\Delta_f |C| &= \frac{1}{2|C|} \Delta_f |C|^2 - \frac{1}{|C|} |\na |C||^2 \\
			&= \frac{1}{2|C|} \left( 2\la \Delta_f C, C \ra + 2 \la \na C, \na C \ra \right) - \frac{1}{|C|} |\na |C||^2 \\
			&= \frac{1}{|C|} \la \Delta_f C, C \ra + \frac{1}{|C|} \left( |\na C|^2 - |\na |C||^2  \right) \\
			&\geq \frac{1}{|C|} \la \Delta_f C, C \ra.
		\end{split}		
	\end{equation}
	Substituting \eqref{3.l1e2} into \eqref{3.l1e1} and using Lemma \ref{lem:ellieqofsolitons}, we obtain
	\begin{equation*}
		\begin{split}
			\Delta_F \frac{|C|}{R} &\geq \frac{1}{R|C|} \la \Delta_f C, C \ra - \frac{|C|}{R^2}\Delta_f R \\
			&= \frac{1}{R|C|} \la C - 2(C^2 +B^tB +2C^{\sharp}), C \ra - \frac{|C|}{R^2}\left( R - 2|Rc|^2\right) \\
			&= \frac{2|C|}{R^2}|Rc|^2 - \frac{1}{R|C|} \la 2(C^2 +B^tB +2C^{\sharp}), C \ra,
		\end{split}
	\end{equation*}
	where $F = f - 2\log R$.
This proves the first inequality in Lemma \ref{lem:inequalityof|C|/R_1}. 	

	Moreover, by the first inequality we just proved and the diagonalization of the matrices $A$ and $C$, we have
	\begin{equation*}
		\begin{split}
			\Delta_F\frac{|C|}{R} &\geq \frac{2|C|}{R^2}|Rc|^2 - \frac{1}{R|C|}\la 2(C^2 + B^tB + 2C^{\sharp}), C \ra \\
			& = \frac{2}{R^2|C|}\left( \frac{1}{4}|C|^2R^2 + |C|^2|\trc|^2 - R\left( \tr(C^3) + \tr(CB^tB) + 2\tr(C^{\sharp}C)  \right)   \right) \\
			&= \frac{2}{R^2|C|}\Bigg( \frac{1}{192}R^4 + \frac{1}{4}R^2\sum_{i=1}^{3}c_i^2 + \frac{1}{48}R^2\sum_{i=1}^{4}\lambda_i^2 + \left( \sum_{i=1}^{3}c_i^2 \right)\left( \sum_{i=1}^{4}\lambda_i^2 \right) \\
			&\hspace{1.5cm}- \frac{3}{12^3}R^4 - \frac{1}{4}R^2\sum_{i=1}^{3}c_i^2 - R\sum_{i=1}^{3}c_i^3 - \frac{1}{48}R^2\sum_{i=1}^{4}\lambda_i^2 \\
			&\hspace{1.5cm}- R\sum_{i=1}^{3}c_i\tilde{b}_i^2 - \frac{6}{12^3}R^4 + \frac{1}{4}R^2\sum_{i=1}^{3}c_i^2 - 6Rc_1c_2c_3	\Bigg).
		\end{split}
	\end{equation*}
	Using (\ref{s2l2.7e2}) and after some cancellations, we obtain
	\begin{equation*}
		\begin{split}
			\Delta_F\frac{|C|}{R} &\geq \frac{2}{R^2|C|}\left( \frac{1}{4}R^2\sum_{i=1}^{3}c_i^2 - 3R\sum_{i=1}^{3}c_i^3 + \left( \sum_{i=1}^{3}c_i^2 \right)\left( \sum_{i=1}^{4}\lambda_i^2 \right) - R\sum_{i=1}^{3}c_i\tilde{b}_i^2 \right) \\
			&= \frac{2}{R^2|C|}\left( \frac{1}{4}R^2\sum_{i=1}^{3}c_i^2 - 3R\sum_{i=1}^{3}c_i^3 + 4\left( \sum_{i=1}^{3}c_i^2 \right)\left( \sum_{i=1}^{3}\tilde{b}_i^2 \right) - R\sum_{i=1}^{3}c_i\tilde{b}_i^2 \right),
		\end{split}		
	\end{equation*}
	where we have used (\ref{s2l2.7e1}) in the last equality.
\end{proof}

Next, we recall an algebraic inequality derived by Li-Ni-Wang \cite{Li-Ni-Wang:18}. Since, in the proof of Theorem \ref{thm:main}, we shall need the equality case that was not stated in \cite{Li-Ni-Wang:18}, we  also include a proof here for the reader's convenience.

\begin{lem} \textup{\bf (Li-Ni-Wang \cite{Li-Ni-Wang:18})} \label{lem:algebraicidentityofRanda}
	Let $(M^4,g,f)$ be a 4-dimensional complete gradient shrinking Ricci soliton with half nonnegative isotropic curvature, then
$$ \frac{1}{4}R^2 \sum_{i=1}^{3}c_i^2 - 3R\sum_{i=1}^{3}c_i^3 \geq 0. $$		
Moreover, the equality holds if and only if either $c_i = 0$ for all $i=1,2,3$, or $c_3=\frac{R}{6}$ and $c_1=c_2=-\frac{R}{12}$.
\end{lem}

\begin{proof}
	We first note that we have the constraints $\sum_{1}^{3}c_i=0$ (as the anti-self dual Weyl curvature $W^-$ is trace free) and $C_1+C_2=\frac{R}{6}+c_1+c_2\geq 0$ (due to the half nonnegative isotropic assumption). It is also easy to see that the second constraint is equivalent to $\frac{R}{6}\geq c_i$, for $1\leq i \leq 3$.
	
Now we define the objective function 
	$$ G(c_1,c_2,c_3)=R\sum_{1}^{3}c_i^2-12\sum_{1}^{3}c_i^3. $$
By using the method of Lagrange multipliers, one finds that the critical points of $G$ satisfy the following equation for some constant $\lambda$:
	$$ \la 2c_1R-36c_1^2, 2c_2R-36c_2^2, 2c_3R-36c_3^2 \ra = \lambda \la 1,1,1 \ra.$$
	Thus, $c_1,c_2,c_3$ are solutions of the quadratic equation $36x^2-2Rx+\lambda = 0.$ By the quadratic formula, we have
	$$ x_{\pm} = \frac{2R\pm \sqrt{4R^2-144\lambda}}{72}. $$
	
	On one hand, if $c_1=c_2=c_3=x_{\pm}$, then $ 0=c_1+c_2+c_3=3\cdot x_{\pm} $, which implies that $c_1=c_2=c_3=0$, hence $G(c_1,c_2,c_3)=0$. 

On the other hand, if $c_1=c_2=x_{-}$ and $c_3=x_+$, then $ 0=c_1+c_2+c_3=2\cdot x_{-} + x_+ $,  implying $ \sqrt{4R^2-144\lambda}=6R $. It then follows that $c_1=c_2=-\frac{R}{18}$, $c_3=\frac{R}{9}$ and  $G(c_1,c_2,c_3)=\frac{R^3}{162} \geq 0$. 

Finally, for the boundary case, we may assume $c_3=\frac{R}{6}$, then $0=c_1+c_2+c_3$ and  $c_2=-\frac{R}{6}-c_1$. Hence,
	\begin{eqnarray*}
		&& G(c_1,c_2,c_3) \\
		&=& R\left( c_1^2 + \left(-\frac{R}{6}-c_1 \right)^2 + \left( \frac{R}{6}\right) ^2 \right) -12 \left( c_1^3 + \left(-\frac{R}{6}-c_1 \right)^3 + \left( \frac{R}{6}\right) ^3 \right)   \\
		&=& 8R\left( c_1+\frac{R}{12}\right)^2 \\
		&\geq& 0,
	\end{eqnarray*}
	with equality if and only if $c_1=c_2=-\frac{R}{12}$ and $c_3=\frac{R}{6}$. 

This finishes the proof of Lemma \ref{lem:algebraicidentityofRanda}.
\end{proof}

\begin{rmk}
	Lemma \ref{lem:inequalityof|C|/R_1} and 
Lemma \ref{lem:algebraicidentityofRanda} also hold for gradient steady and expanding Ricci solitons.
\end{rmk}

\bigskip
\noindent {\bf Conclusion of the Proof of Theorem \ref{thm:main}.}
By Remark 2.2, it suffices to assume $R>0$ everywhere so Lemma \ref{lem:inequalityof|C|/R_1}  applies.  
 
Since the Ricci tensor $Rc$ has an eigenvalue with multiplicity $3$ by assumption, from (\ref{matrixB}) we know that either $B^tB = 0$, or $B^tB = b^2\ \!\text{Id}$ for some constant $b$. In either case, we have
$$ R\sum_{i=1}^{3}c_i\tilde{b}_i^2 = 0. $$
Thus, it follows from Lemma \ref{lem:inequalityof|C|/R_1} and Lemma \ref{lem:algebraicidentityofRanda} that $\Delta_F(|C|R^{-1}) \geq 0$. 

Now, we are going to apply the Yau-Naber Liouville maximum principle (Lemma \ref{lem:Yau-NaberLiouville}) with $u=|C|R^{-1}$ and $h=F$ to get a pinching estimate on the anti-self-dual Weyl curvature $W^{-}$. On one hand, by Lemma \ref{lem:M-S} (or Lemma 2.2 and Lemma 2.5), we have
\begin{equation} \label{F-vol}
\int_M e^{-F} = \int_M R^2e^{-f} < \infty. 
\end{equation}
On the other hand, we note that half nonnegative isotropic curvature implies
\begin{equation} \label{boundfor Ci}
 -\frac{R}{4} \leq C_1 \leq C_2 \leq C_3 \leq \frac{R}{4}. 
\end{equation} 
Thus $|C|^2 \leq \frac{3}{16}R^2$, from which we get
$$ \int_M \frac{|C|^2}{R^2}e^{-F} = \int_M |C|^2 e^{-f} \leq \frac{3}{16} \int_M R^2 e^{-f} < \infty. $$
Therefore, by applying Lemma \ref{lem:Yau-NaberLiouville}, we conclude that $|C|/R$ is a constant.

Now, using $|C|R^{-1} \equiv \text{constant}$ and Lemma \ref{lem:inequalityof|C|/R_1}, it follows that
$$  \frac{1}{4}R^2\sum_{i=1}^{3}c_i^2 - 3R\sum_{i=1}^{3}c_i^3 + 4\left( \sum_{i=1}^{3}c_i^2 \right)\left( \sum_{i=1}^{3}\tilde{b}_i^2 \right) = 0.  $$
By Lemma \ref{lem:algebraicidentityofRanda} and the equation above, we see that  either $c_i = 0$ for $i=1,2,3$, or $\tilde{b}_i=0$ for  $i=1,2,3$ and $c_3=\frac{R}{6}$, $c_1=c_2=-\frac{R}{12}$. 

Recall that $c_1\le c_2\le c_3$ are the eigenvalues of $W^-$. Hence, in the first case when $c_i = 0$ ($1\leq i\leq 3$), $(M^4, g)$ is half locally conformally flat. Thus, by the work of \cite{Chen-Wang:15} or \cite{Cao-Chen:13}, $(M^4, g, f)$ is either $\rS^4$, or $\CP^2$, or a finite quotient of $\rS^3\times \R$. In the second case, in view of the condition $\tilde{b}_i=0$ ($1\le i\le3$) and (\ref{s2l2.7e1}), $(M^4, g)$ is Einstein with half nonnegative isotropic curvature. Also, since $c_1=c_2=-\frac{R}{12}$, it is not half PIC. Then, applying the classification results of Richard-Sechadri \cite{Richard-Seshadri:16} (see also \cite{Wu:17}), we  conclude that $(M^4, g,f)$ is either the Gaussian soliton on $\R^4$ or K\"ahler Einstein with nonnegative isotropic curvature (but not half PIC). By Corollary \ref{thm:kahler}, the latter must be $\CP^1\times \CP^1$. In particular, if $(M^4,g)$ has half PIC, then the second case is excluded. Therefore, we have completed the proof of Theorem \ref{thm:main}.

\hfill $\Box$

\begin{rmk} 
Alternatively, one can apply the Yau-Naber Liouville maximum principle to the quantity $(C_3-C_1){R}^{-1}$ to get a slightly different proof of Theorem \ref{thm:main} given below.  
\end{rmk}

\begin{proof}
	By direct computations, for $R>0$ and $F = f - 2\log R$, we have
	$$ \Delta_F \frac{C}{R} = \frac{2}{R^2}[ C|Rc|^2 - R(C^2 + B^tB + 2C^{\sharp}) ].  $$
	Hence,
	$$ \Delta_F \frac{C_1}{R} \leq \frac{2}{R^2}[ C_1|Rc|^2 - R(C_1^2 + B_1^2 + 2C_2C_3) ],  $$
	and
	$$ \Delta_F \frac{C_3}{R} \geq \frac{2}{R^2}[ C_3|Rc|^2 - R(C_3^2 + B_3^2 + 2C_1C_2) ].  $$
	Then, it follows that 

	\begin{eqnarray*}
		\Delta_F \frac{C_3-C_1}{R} &\geq& \frac{2}{R^2} [ (C_3-C_1)|Rc|^2  + R(C_1^2 + B_1^2 + 2C_2C_3 - C_3^2 - B_3^2 - 2C_1C_2) ] \\
		&=& \frac{2}{R^2} \left[ (C_3-C_1)\left( |\mathring{Rc}|^2 + \frac{1}{4}R^2\right)  + R(C_1^2 - C_3^2) \right] \\
		&& + \frac{2}{R^2} \left[ R(B_1^2 -B_3^2) + 2RC_2(C_3-C_1) \right] \\
		&=& \frac{2}{R^2} \left[ (C_3-C_1) |\mathring{Rc}|^2 + R(B_1^2 -B_3^2) \right] \\
		&& + \frac{2}{R^2} \left[ R(C_3-C_1)\left(  \frac{1}{4}R - (C_3+C_1) + 2C_2 \right) \right] \\
		&=& \frac{2}{R^2} \left[ (C_3-C_1) |\mathring{Rc}|^2 + R(B_1^2 -B_3^2) + 3R(C_3-C_1)C_2 \right].
	\end{eqnarray*}
	
	Under the assumption that the Ricci tensor has an eigenvalue with multiplicity 3, from (\ref{matrixB}) we know that either $B^tB = 0$ or $B^tB = b^2\!\ \text{Id}$ for some constant $b$. In any case, we have
	$$ B_1=B_2=B_3. $$
	Therefore, as half nonnegative isotropic curvature implies $C_2\geq 0$, we obtain
	\begin{eqnarray} \label{eq:C3minusC1}
		\Delta_F \frac{C_3-C_1}{R} = \frac{2}{R^2} \left[ (C_3-C_1) |\mathring{Rc}|^2 + 3R(C_3-C_1)C_2 \right] \geq 0.
	\end{eqnarray}
	On the other hand, from (\ref{boundfor Ci}) we get
	$$ \int_M \frac{|C_3-C_1|^2}{R^2}e^{-F} = \int_M |C_3-C_1|^2 e^{-f} \leq \frac{1}{4} \int_M R^2 e^{-f} < \infty. $$
	Hence, by (\ref{F-vol}) and applying the Yau-Naber Liouville theorem again, we conclude that $(C_3-C_1)R^{-1}$ is a constant.
	
	Now, by the fact that $(C_3-C_1)R^{-1} \equiv \text{constant}$ and (\ref{eq:C3minusC1}), we have
	$$ (C_3-C_1) |\mathring{Rc}|^2 + 3R(C_3-C_1)C_2 = 0, $$
	which implies that either $C_3=C_1$, or $\mathring{Rc}=0$ and $C_2=0$. In the first case, when $C_3=C_1$, it follows that $W^{-}\equiv 0$ and $(M^4, g)$ is half locally conformally flat. In the second case, $(M^4, g)$ is Einstein with half nonnegative isotropic curvature. 
	So Theorem \ref{thm:main}  follows as before. 
\end{proof}

\bigskip

\end{document}